\documentclass[a4paper]{amsart}

\usepackage{amsmath,amssymb,amsthm}
\usepackage[all]{xy}
\usepackage{eucal}
\usepackage[dvipdfmx]{graphicx}
\usepackage{graphics}
\usepackage{tikz}
\usepackage[dvipdfmx,colorlinks=true,backref=page]{hyperref}

\newtheorem{theorem}{Theorem}[section]
\newtheorem{proposition}[theorem]{Proposition}
\newtheorem{lemma}[theorem]{Lemma}
\newtheorem{corollary}[theorem]{Corollary}

\theoremstyle{definition}

\theoremstyle{remark}

\numberwithin{equation}{section}

\newcommand{\Z}{\mathbb{Z}}
\newcommand{\Q}{\mathbb{Q}}

\newcommand{\C}{\mathbb{C}}
\newcommand{\G}{\mathcal{G}}
\newcommand{\map}{\operatorname{Map}}
\newcommand{\odd}{{\operatorname{odd}}}
\newcommand{\rat}{_{\mathcal{R}}}

\title[The classifying spaces of $SO(n)$-gauge groups over $S^2$]{On the cohomology of the classifying spaces of $SO(n)$-gauge groups over $S^2$}

\author{Yuki Minowa}
\address{Department of Mathematics, Kyoto University, Kyoto, 606-8502, Japan}
\email{minowa.yuki.48z@st.kyoto-u.ac.jp}

\begin{document}

\begin{abstract}
  Let $\G_{\alpha}(X, G)$ be the $G$-gauge group over a space $X$ corresponding to a map $\alpha \colon X \to BG$. We compute the integral cohomology of $B\G_{1}(S^2, SO(n))$ for $n = 3,4$. We also show that the homology of $B\G_{1}(S^2, SO(n))$ is torsion free if and only if $n\le 4$. As an application, we classify the homotopy types of $SO(n)$-gauge groups over a Riemann surface for $n\le 4$.
\end{abstract}

\maketitle


\section[Introduction]{Introduction}
 \label{Intro}

Let $G$ be a topological group, and let $P$ be a principal $G$-bundle over a base space $X$. Then automorphisms of $P$ are, by definition, $G$-equivariant self-maps of $P$ covering the identity map of $X$. The \emph{gauge group} of $P$ is defined as the topological group of automorphisms of $P$.
Let $\G_{\alpha}(X, G)$ denote the gauge group of the principal $G$-bundle over $X$ corresponding to a map $\alpha \colon X \to BG$.

Gauge groups have been studied in several contexts of geometry, topology, and physics. Here, we briefly recall the homotopy theory of gauge groups. Given $G$ and $X$, we have a family of principal $G$-bundles over $X$, and so we get a family of their gauge groups. There are many previous studies about the classification of the homotopy types of gauge groups in this family. As shown by Gottlieb \cite{G72}, there is a natural homotopy equivalence
\begin{equation}
  \label{mapping space}
  B\G_{\alpha}(X, G) \simeq \map(X, BG; \alpha)
\end{equation}
where $\map(X,Y; f)$ denotes the path-component of the space of maps from $X$ to $Y$ containing a map $f\colon X \to Y$ (see also \cite{AB83}). Then the classification of the homotopy types of gauge groups is closely related to the classical problem in algebraic topology, the classification of components of a mapping space \cite{S10}. There are some general results on the classification of the homotopy types of gauge groups \cite{KameKT, KishiK10, KishiKT14, Th10A, Th19}. Remarkably, Crabb and Sutherland \cite{CrS00} showed that there are only finitely many homotopy types of $\G_{\alpha}(X, G)$ as $\alpha$ ranges over all maps $X \to BG$ whenever $G$ is a compact connected Lie group and $X$ is a connected finite complex, while there are possibly infinitely many principal $G$-bundles over $X$.

The classification of the homotopy types of gauge groups is generalized to the classification of the $A_n$-types: Tsutaya \cite{Tsut12} showed that there are only finitely many $A_n$-types for $n<\infty$ whenever $G$ is a compact connected Lie group and $X$ is a finite complex. On the other hand, Kishimoto and Tsutaya \cite{KishiTsuk16} showed that there are infinitely many $A_\infty$-types, i.e. the homotopy types of the classifying spaces, of gauge groups whenever $G$ is a compact connected simple Lie group and $X=S^d$ with $\pi_{d-1}(G)\otimes\Q\ne 0$. In addition, there are many results on the classification for specific $G$ and $X$ \cite{Cu18, HamK06, HamK07, HasKKS16, KamiKKKT07, KishiK19, KishiThT17, Ko91, Th10B, Th15}.

Although there are many results on the classification of the homotopy types of gauge groups, not much is known about their (co)homology. In particular, there are only a few results determining the (co)homology of the classifying spaces of gauge groups. Takeda \cite{Ta21} computed the integral cohomology ring of $B\G_{k}(S^2, U(n))$. Kishimoto and Theriault \cite{KishiTh22} proved that, for a compact connected simple Lie group $G$ of type $(n_1, \ldots, n_l)$ and a prime $p$ larger than $n_l+1$, there is an isomorphism
\[
H_*(B\G_{k}(S^4, G); \Z/p)
\cong H_*(BG; \Z/p) \otimes H_*(\Omega^3 G \langle 3 \rangle; \Z/p)
\]
as $\Z/p$-vector spaces and so the mod $p$ homology of $B\G_{k}(S^4, G)$ does not depend on the classifying map $k \in \Z \cong \pi_3(G)$. Partial results on the mod $p$ (co)homology of $B\G_{\alpha}(X, G)$ for specific $G$ and $X$ are given by Choi \cite{Ch06}, Masbaum \cite{M91} and Tsukuda \cite{Tsuk97}. As of today, Tsukuda's result \cite{Tsuk97} is the only nontrivial result on the integral or mod $p$ (co)homology of $B\G_{\alpha}(S^d, G)$ for a Lie group $G$ that has $p$-torsion in its homology.

In this paper, we study the cohomology of $B\G_k(S^2,SO(n))$ for $n\ge 3$, where $k=0,1\in\Z/2\cong\pi_1(SO(n))$. We will always assume $n\ge 3$ as long as we consider $SO(n)$. First, we will apply the result of Takeda \cite{Ta21} on the integral cohomology of $B\G_k(S^2,U(n))$ for determining the integral cohomology of $B\G_1(S^2,SO(n))$ for $n\le 4$, where the result of Takeda includes a minor mistake and we will correct it. To state the result, we set notation. Let $A$ be a graded abelian group of finite type. The \emph{rational closure} of a subgroup $B$ of $A$ is defined by
\[
  B_\mathcal{R}=\{a\in A\mid na\in B\text{ for some }0\ne n\in\Z\}.
\]
We also define a polynomial $s_n\in\Z[x_1,\ldots,x_n]$ inductively by
\begin{equation}
  \label{Newton}
  s_0=1,
  \quad s_1=x_1
  \quad \text{and}
  \quad s_n=(-1)nx_n+\sum_{i=1}^{n-1}(-1)^{n+i+1}x_{n-i}s_i.
\end{equation}
\noindent
So if $x_i$ is the $i$-th elementary symmetric function in $t_1,\ldots,t_n$, then for $k\le n$, we have
\[
  s_k=t_1^k+\cdots+t_n^k.
\]
Now we are ready to state the first result.

\begin{theorem}
  \label{main1}
  \begin{enumerate}
    \item There is an isomorphism
    \[
      H^*(B\G_1(S^2,SO(3));\Z)\cong\Z[c_1,c_2,x_1,x_2,\ldots]/((h_2,h_3,\ldots)\rat+(x_1))
    \]
    where $|c_i| = |x_i| = 2i$ and $h_i=s_i-s_{i-1}c_1+s_{i-2}c_2$.

    \item There is an isomorphism
    \[
      H^*(B\G_1(S^2,SO(4));\Z)\cong\Z[c_1,c_2,x_1,x_2,\ldots]\otimes\Z[c_1,c_2,x_1,x_2,\ldots]/I
    \]
    where $|c_i| = |x_i| = 2i$, $h_i=s_i-s_{i-1}c_1+s_{i-2}c_2$ and
    \[
      I=(h_2\otimes 1,1\otimes h_2,h_3\otimes 1,1\otimes h_3,\ldots)\rat+(c_1\otimes 1-1\otimes c_1,x_1\otimes 1).
    \]
  \end{enumerate}
\end{theorem}

Next, as an application of Theorem \ref{main1}, we will generalize the result of Tsukuda \cite{Tsuk97}, in which he showed that the homology of the classifying space $B\G_1(S^2,SO(3))$ is torsion free.

\begin{theorem}
  \label{main2}
  The homology of $B\G_1(S^2,SO(n))$ is torsion free if and only if $n\le 4$.
\end{theorem}

Finally, as another application of Theorem \ref{main1}, we will compute $\pi_1(\G_k(S^2,SO(n)))$ for $n\le 4$. Then by combining with Theriault's homotopy decomposition of a gauge group over a Riemann surface \cite{Th10A}, we can determine $\pi_1(\G_k(\Sigma_g,SO(n)))$ for $n\le 4$ as follows, where $\Sigma_g$ denotes a Riemann surface of genus $g$. Note that the homotopy set $[\Sigma_g,BSO(n)]$ is isomorphic to $\Z/2$.

\begin{theorem}
  \label{main3}
  For $n\le 4$, there are isomorphisms
  \[
  \pi_1(\G_k(\Sigma_g,SO(n)))
  \cong
  \begin{cases}
    \Z^{n-2} \oplus \Z/2 &(k = 0) \\
    \Z^{n-2} &(k = 1).
  \end{cases}
  \]
\end{theorem}

Then we get the following immediate corollary.

\begin{corollary}
  For $n\le 4$, $\G_k(\Sigma_g,SO(n))\simeq\G_l(\Sigma_g,SO(n))$ if and only if $k=l$.
\end{corollary}

This paper is structured as follows. Section \ref{Spinc} computes the integral cohomology of $B\G_k(S^2,Spin^c(n))$ for $n=3,4$ by correcting the result of Takeda \cite{Ta21} on the integral cohomology of $B\G_k(S^2,U(n))$. Section \ref{Coh} computes the integral cohomology ring of $B\G_1(S^2,SO(n))$ for $n=3, 4$ and proves Theorem \ref{main1}. Section \ref{Tor} computes the mod $2$ cohomology of $B\G_1(S^2,SO(n))$ for $n \ge 5$ and gives a proof of Theorem \ref{main2}. Finally, Section \ref{Sigmag} proves Theorem \ref{main3} as an application of the proof of Theorem \ref{main1}.


\section[Cohomology]{Cohomology of $B\G_k(S^2,Spin^c(n))$ for $n=3,4$}
  \label{Spinc}

This section computes the integral cohomology of $B\G_k(S^2,Spin^c(n))$ for $n=3,4$ by correcting the result of Takeda \cite{Ta21} on the integral cohomology of $B\G_k(S^2,U(n))$. Recall that $Spin^c(n)$ is defined as the nontrivial extension of $SO(n)$ by $S^1$:
\[
1 \to S^1 \xrightarrow{\alpha} Spin^c(n) \xrightarrow{\pi} SO(n) \to 1.
\]
This short exact sequence of Lie groups is natural with respect to the inclusions $i'_n\colon Spin^c(n) \to Spin^c(n+1)$ and $i_n\colon SO(n) \to SO(n+1)$. Then, in particular, there are exceptional isomorphisms:
\[
  Spin^c(3)\cong U(2)\quad\text{and}\quad
  Spin^c(4)\cong
  \left\{
  (A,B) \in U(2) \times U(2) \mid \det A = \det B
  \right\}
\]
where $\alpha \colon S^1 \to Spin^c(3)$ and $i'_3 \colon Spin^c(3) \to Spin^c(4)$ are given by
\[\alpha(e^{i\theta}) =
\begin{pmatrix}
  e^{i\theta} & 0 \\
  0 & e^{i\theta}
\end{pmatrix},
\quad
i'_3(A) = (A,A).
\]
So we can get the cohomology of $B\G_k(S^2,Spin^c(n))$ for $n=3,4$ from that of $B\G_k(S^2,U(2))$.

Takeda \cite{Ta21} described the integral cohomology of $B\G_k(S^2,U(n))$, but it includes a mistake. Then we correct it. Bott \cite{B58} studied the homology of the loop space of a Lie group, and in particular, we have

\begin{proposition}
  \label{Bott}
  There is an isomorphism
  \[
    H^*(\Omega_0 U(n);\Z) \cong \Z[x_1, x_2,\ldots]
    /(s_n,s_{n+1},\ldots)\rat,\quad |x_i|=2i
  \]
  where $s_n$ is as in \eqref{Newton}.
\end{proposition}

\noindent
By \eqref{mapping space}, we have $B\G_k(S^2,U(n)) \simeq \map(S^2,BU(n); k)$. Then there is a homotopy fibration
\begin{equation}
  \label{fibration U(n)}
  \Omega_0U(n)\xrightarrow{\iota}B\G_k(S^2,U(n))\xrightarrow{e}BU(n)
\end{equation}
given by the evaluation at the basepoint of $S^2$. Recall that the cohomology of $BU(n)$ is given by
\[
  H^*(BU(n);\Z) \cong \Z[c_1,\ldots,c_n],\quad|c_i|=2i
\]
where $c_i$ is the universal $i$-th Chern class. Then
\[
H^{\odd}(BU(n);\Z) = 0 \quad \text{and} \quad H^{\odd}(\Omega_0 U(n);\Z) = 0
\]
by Proposition \ref{Bott}. Thus the integral cohomology Serre spectral sequence for the above homotopy fibration collapses at the $E_2$-term and we get

\begin{lemma}
  \label{U(n) module}
  The integral cohomology of $B\G_k(S^2,U(n))$ is a free abelian group and generated by $c_i=e^*(c_i)$ for $i=1,\ldots,n$ and $x_j$ for $j\ge 1$ as an algebra, where $\iota^*(x_j)=x_j$ in Proposition \ref{Bott}.
\end{lemma}

\noindent
Let $c_0 = 1$ and $c_i = 0$ for $i > n$. We define
\[
  h_i=kc_i+\sum_{j=1}^i(-1)^js_jc_{i-j} \quad \text{for} \quad i \geq 1.
\]
Takeda \cite{Ta21} showed that $h_i=0$ in $H^*(B\G_k(S^2,U(n));\Z)$ for $i\ge n$ by constructing the element $x_j \in H^{2j}(B\G_k(S^2,U(n));\Z)$ explicitly. Thus by Lemma \ref{U(n) module}, there is a surjection
\[
  \Z[c_1,\ldots,c_n,x_1,x_2,\ldots]/(h_n,h_{n+1},\ldots)\to H^*(B\G_k(S^2,U(n));\Z).
\]
Takeda \cite{Ta21} claimed that the surjection above is an isomorphism, which is false. To see this, let $n = 2$ and $k = 1$. Then
\[
h_4 - c_1 h_3 - (c_1 x_1 + x_1^2) h_2 = 2y
\]
where
\[
y = c_1^2 x_2 + c_1 x_1 x_2 - x_1^2 x_2 - c_2 x_2 + x_2^2 + 2x_1 x_3 - 2x_4.
\]
However, since the coefficient of $x_4$ in $f \in (h_2,h_3,\ldots)$ is a multiple of 4, we have $y \notin (h_2,h_3,\ldots)$. Thus the image of $y$ in $\Z[c_1,c_2,x_1,x_2,\ldots]/(h_2,h_3,\ldots)$ is of order $2$, while by Lemma \ref{U(n) module},  $H^*(B\G_1(S^2,U(2));\Z)$ is a free abelian group.

Note that, for a free abelian group $A$ and its subgroup $B$, $B\rat$ is the minimal direct summand of $A$ containing $B$. Then $\Z[c_1,\ldots,c_n,x_1,x_2,\ldots]/(h_n,h_{n+1},\ldots)\rat$ is the maximal free abelian subgroup of $\Z[c_1,\ldots,c_n,x_1,x_2,\ldots]/(h_n,h_{n+1},\ldots)$. Thus by Lemma \ref{U(n) module}, we have a surjection
\begin{equation}
  \label{projection}
  \Z[c_1,\ldots,c_n,x_1,x_2,\ldots]/(h_n,h_{n+1},\ldots)\rat\to H^*(B\G_k(S^2,U(n)); \Z).
\end{equation}

\begin{theorem}
 \label{BU}
  There is an isomorphism
  \[
    H^*(B\G_k(S^2,U(n));\Z)\cong\Z[c_1,\ldots,c_n,x_1,x_2,\ldots]/(h_n,h_{n+1},\ldots)\rat.
  \]
\end{theorem}

\begin{proof}
  Since the rational cohomology Serre spectral sequence for the homotopy fibration \eqref{fibration U(n)} collapses at the $E_2$-term as above, it follows from Proposition \ref{Bott} that the Poincar\'e series of the rational cohomology of $B\G_k(S^2,U(n))$ is given by
  \[
    \prod_{i=1}^n\frac{1}{1-t^{2i}}\prod_{i=1}^{n-1}\frac{1}{1-t^{2i}}.
  \]
  On the other hand, one can easily see that $h_n, h_{n+1}, \ldots$ is a regular sequence in $\Q[c_1,\ldots,c_n,x_1,x_2,\ldots]$ and so $(\Z[c_1,\ldots,c_n,x_1,x_2,\ldots]/(h_n,h_{n+1},\ldots)\rat)\otimes\Q$ has the same Poincar\'e series as $B\G_k(S^2,U(n))$. Thus the map \eqref{projection} is injective, hence an isomorphism.
\end{proof}

\begin{corollary}
  \label{BU2}
  There is an isomorphism
    \begin{align*}
      H^*(B\G_k(S^2, Spin^c(3)); \Z)
      &\cong H^*(B\G_k(S^2, U(2)); \Z) \\
      &\cong \Z[c_1, c_2, x_1, x_2, \ldots]/ (h_2, h_3, \ldots)\rat.
    \end{align*}
\end{corollary}
\noindent

Next, we compute the integral cohomology of $B\G_k(S^2,Spin^c(4))$. Recall that there is a short exact sequence of Lie groups
\[
1 \to Spin^c(4) \xrightarrow{\alpha} U(2)\times U(2) \xrightarrow{\pi} S^1 \to 1
\]
where $\alpha$ is the inclusion and $\pi(A, B) = \det(AB^{-1})$. We denote the composite map $S^2 \xrightarrow{k} BSpin^c(4) \xrightarrow{\alpha} BU(2)\times BU(2) \xrightarrow{p} BU(2)$ by $\underline{k}$, where $p$ is the first projection. Since the induced map $\Omega \alpha \colon \Omega_0 Spin^c(4) \to \Omega_0 (U(2)\times U(2))$ is a homotopy equivalence, we have

\begin{lemma}
  \label{diag1}
 There is a homotopy commutative diagram
  \[
  \xymatrix{
    \Omega_0 Spin^c(4) \ar[r]^-{\simeq} \ar[d]^{\iota}
    & \Omega_0 (U(2) \times U(2)) \ar[r] \ar[d]^{\iota}
    & \ast \ar[d] \\
    B\G_k(S^2, Spin^c(4)) \ar[r]^-{\alpha_*} \ar[d]^e
    & B\G_{(\underline{k},\underline{k})}(S^2, U(2)\times U(2)) \ar[r]^-{\pi_*} \ar[d]^e
    & B\G_0(S^2, S^1) \ar@{}[d]|{\rotatebox{-90}{$\simeq$}} \\
    BSpin^c(4) \ar[r]^-{\alpha} & BU(2) \times BU(2) \ar[r]^-{\pi} & BS^1
    }
  \]
where all columns and rows are homotopy fibrations.
\end{lemma}

\begin{lemma}
  \label{BSpinc4}
  $H^*(BSpin^c(4);\Z)$ is a free abelian group and there is an isomorphism
  \[
    H^*(BSpin^c(4);\Z)
    \cong \Z[c_1,c_2]\otimes\Z[c_1,c_2]/(c_1 \otimes 1 - 1 \otimes c_1).
  \]
\end{lemma}

\begin{proof}
  First, we consider $H^*(BSpin^c(4);\Z)$. Consider the short exact sequence of Lie groups
  \[
  1 \to SU(2) \times SU(2) \hookrightarrow Spin^c(4) \to S^1 \to 1.
  \]
  Then there is a homotopy fibration
  \[
  BSU(2) \times BSU(2) \to BSpin^c(4) \to BS^1.
  \]
  Since
  \[
  H^{\odd}(BSU(2) \times BSU(2);\Z) = 0
  \quad \text{and} \quad H^{\odd}(BS^1;\Z) = 0,
  \]
  the integral cohomology Serre spectral sequence for the homotopy fibration collapses at the $E_2$-term by degree reasons. Thus we get $H^{\odd}(BSpin^c(4);\Z) = 0$. Since both $H^*(BSU(2) \times BSU(2);\Z)$ and $H^*(BS^1;\Z)$ are free abelian groups, $H^*(BSpin^c(4);\Z)$ is also a free abelian group.
  Next, we consider the homotopy fibration of the bottom row of the diagram in Lemma \ref{diag1}
  \[
    BSpin^c(4) \xrightarrow{\alpha} BU(2)\times BU(2) \xrightarrow{\pi} BS^1.
  \]
  Recall that $H^{\odd}(BSpin^c(4);\Z) = 0$. By degree reasons, the integral cohomology Serre spectral sequence for the homotopy fibration collapses at the $E_2$-term. Since $\pi^*(c_1) = c_1 \otimes 1 - 1 \otimes c_1$, we get the isomorphism in the statement.
\end{proof}

\begin{proposition}
  \label{BGSpinc4}
  $H^*(B\G_k(S^2, Spin^c(4));\Z)$ is a free abelian group and there is an isomorphism
  \[
    H^*(B\G_k(S^2, Spin^c(4));\Z)
    \cong
    \Z[c_1,c_2,x_1,x_2,\ldots]\otimes\Z[c_1,c_2,x_1,x_2,\ldots]/I
  \]
  where
  \[
    I=(h_2\otimes 1,1\otimes h_2,h_3\otimes 1,1\otimes h_3,\ldots)_\mathcal{R}+(c_1\otimes 1-1\otimes c_1).
  \]
\end{proposition}

\begin{proof}
Recall that
\[
H^{\odd}(\Omega_0 Spin^c(4);\Z) \cong H^{\odd}(\Omega S^3 \times \Omega S^3;\Z) = 0.
\]
By Lemma \ref{BSpinc4}, $H^{\odd}(BSpin^c(4);\Z) = 0$. Then we get
\[
  H^{\odd}(B\G_k(S^2, Spin^c(4)); \Z) = 0,
\]
since the integral cohomology Serre spectral sequence for the homotopy fibration of the left column of the diagram in Lemma \ref{diag1}
\[
\Omega_0 Spin^c(4) \xrightarrow{\iota} B\G_k(S^2, Spin^c(4)) \xrightarrow{e} BSpin^c(4)
\]
collapses at the $E_2$-term by degree reasons. Since both $H^*(\Omega_0 Spin^c(4);\Z)$ and $H^*(BSpin^c(4);\Z)$ are free abelian groups, $H^*(B\G_k(S^2, Spin^c(4));\Z)$ is also a free abelian group.
Then, since $H^{\odd}(BS^1;\Z) = 0$, the integral cohomology Serre spectral sequence for the homotopy fibration of the middle row of the diagram in Lemma \ref{diag1}
\[
B\G_k(S^2, Spin^c(4)) \xrightarrow{\alpha_*} B\G_{(\underline{k},\underline{k})}(S^2, U(2)\times U(2)) \xrightarrow{\pi_*} BS^1
\]
collapses at the $E_2$-term by degree reasons. Furthermore, we have
\[
  (\pi_*)^*(c_1)
  = (\pi \circ e)^*(c_1)
  = e^* \circ \pi^*(c_1)
  = e^* (c_1 \otimes 1 - 1 \otimes c_1)
  = c_1 \otimes 1 - 1 \otimes c_1.
\]
Thus we get
\begin{align*}
  &H^*(B\G_k(S^2, Spin^c(4));\Z) \\
  &\cong H^*(B\G_{(\underline{k},\underline{k})}(S^2, U(2)\times U(2));\Z) / ((\pi_*)^*(c_1)) \\
  &\cong H^*(B\G_{k}(S^2, U(2));\Z) \otimes H^*(B\G_{k}(S^2, U(2));\Z) / (c_1 \otimes 1 - 1 \otimes c_1) \\
  &\cong \Z[c_1,c_2,x_1,x_2,\ldots]\otimes\Z[c_1,c_2,x_1,x_2,\ldots]/I
\end{align*}
\noindent
and the proof is finished.
\end{proof}
\noindent


\section[Computation of cohomology rings]{Cohomology of $B\G_1(S^2,SO(n))$ for $n=3, 4$}
  \label{Coh}

This section computes the integral cohomology ring of $B\G_1(S^2,SO(n))$ for $n=3, 4$ and proves Theorem \ref{main1}. Recall that there is an exact sequence of Lie groups
\[
1 \to S^1 \xrightarrow{\alpha} Spin^c(n) \xrightarrow{\pi} SO(n) \to 1.
\]
For an integer $k$, let $\overline{k}$ be its mod $2$ reduction. Since the induced map $\Omega \pi \colon \Omega_0 Spin^c(n) \to \Omega_0 SO(n)$ is a homotopy equivalence, we have

\begin{lemma}
  \label{diag2}
There is a homotopy commutative diagram
  \[
    \xymatrix{
      \ast \ar[r] \ar[d] & \Omega_0 Spin^c(n) \ar[r]^{\simeq} \ar[d]^{\iota} & \Omega_0 SO(n) \ar[d]^{\iota}\\
      B\G_0(S^2, S^1) \ar[r]^-{\alpha_*} \ar@{}[d]|{\rotatebox{-90}{$\simeq$}} & B\G_k(S^2, Spin^c(n)) \ar[r]^-{\pi_*} \ar[d]^{e} & B\G_{\overline{k}}(S^2, SO(n)) \ar[d]^{e} \\
      BS^1 \ar[r]^-{\alpha} & BSpin^c(n) \ar[r]^-{\pi} & BSO(n)
    }
  \]
where all columns and rows are homotopy fibrations.
\end{lemma}
\noindent

First, we compute the $n=3$ case. By Lemma \ref{diag2}, there is a homotopy fibration
\[
  BS^1
  \xrightarrow{\alpha_*}B\G_k(S^2,Spin^c(3))
  \xrightarrow{\pi_*}B\G_{\overline{k}}(S^2,SO(3)).
\]
Let $c_i, x_i \in H^{2i}(B\G_k(S^2, Spin^c(3));\Z)$ be the generators as in Corollary \ref{BU2} and $c_1 \in H^2(BS^1;\Z)$ be the first universal Chern class. Recall that, in the short exact sequence of Lie groups
\[
1 \to S^1 \xrightarrow{\alpha} Spin^c(3) \cong U(2) \xrightarrow{\pi} SO(3) \to 1,
\]
$\alpha \colon S^1 \to U(2)$ is given by $\alpha(e^{i\theta}) =
\begin{pmatrix}
  e^{i\theta} & 0 \\
  0 & e^{i\theta}
\end{pmatrix}$. Then we can take the generators $c_i \in H^{2i}(BSpin^c(3);\Z)$ to satisfy ${\alpha}^*(c_1) = 2 c_1$ and ${\alpha}^*(c_2) = c_1^2$. By applying Lemma \ref{U(n) module}, the following statement holds.

\begin{lemma}
  \label{alphac}
  ${(\alpha_*)}^*(c_1) = 2 c_1$ and ${(\alpha_*)}^*(c_2) = c_1^2$.
\end{lemma}

\begin{proposition}
  \label{alpha}
  ${(\alpha_*)}^*(x_1) = kc_1$.
\end{proposition}

\begin{proof}
  Since $c_1$ generates $H^2(BS^1;\Z)$, there exists an integer $l$ such that
  \[
    {(\alpha_*)}^*(x_1) = (k-l)c_1.
  \]
  By Corollary \ref{BU2}, we have
  \begin{align*}
  kc_2 - s_1 c_1 + s_2 &= 0 \\
  s_{n-2}c_2 - s_{n-1}c_1 + s_{n} &= 0 \quad \text{for} \quad n \ge 3
  \end{align*}
  \noindent
  in $H^{2i}(B\G_k(S^2,Spin^c(3));\Z)$.
  By Lemma \ref{alphac}, we have
  \[
  {(\alpha_*)}^*(c_1) = 2 c_1
  \quad \text{and} \quad {(\alpha_*)}^*(c_2) = c_1^2.
  \]
  Then we get
  \[
    {(\alpha_*)}^*(s_n) = (k-nl)c_1^n
  \]
  by induction. Recall the Girard-Newton formula
  \[
    s_n
    = (-1)^n n
    \sum_{\substack{r_1 + 2r_2 + \dots + nr_n = n \\ r_1 \geq 0, \dots, r_n \geq 0}}
    \frac{(r_1 + r_2 + \dots + r_n - 1)!}{r_1!r_2! \dots r_n!}
    \prod_{i = 1}^n (-x_i)^{r_i}.
  \]
  Then $s_p - x_1^p$ is divisible by $p$ for $p$ prime. On the other hand,
  \[
  {(\alpha_*)}^*(s_p - x_1^p) = ((k-pl)-(k-l)^p)c_1^p = (((k-l)-(k-l)^p)-(p-1)l)c_1^p
  \]
  is divisible by $p$ if and only if $l$ is divisible by $p$. Then $l$ is divisible by all primes and so we get $l=0$. Thus
  \[
    {(\alpha_*)}^*(x_1) = k c_1
  \]
  and the proof is finished.
\end{proof}
\noindent
By Proposition \ref{alpha}, we have ${(\alpha_*)}^*(x_1) = c_1$ for $x_1 \in H^{2}(B\G_1(S^2,Spin^c(3));\Z)$.
Now we consider the homotopy fibration
\[
BS^1
\xrightarrow{\alpha_*}B\G_1(S^2,Spin^c(3))
\xrightarrow{\pi_*}B\G_1(S^2,SO(3))
\]
and prove the first half of Theorem \ref{main1}.

\begin{proof}[Proof of Theorem \ref{main1} (1)]
  By the Leray-Hirsch theorem, we get an isomorphism of $H^*(B\G_1(S^2,SO(3));\Z)$-modules
  \[
  \Phi \colon H^*(BS^1;\Z) \otimes H^*(B\G_1(S^2,SO(3));\Z) \to H^*(B\G_1(S^2,Spin^c(3));\Z)
  \]
  defined by
  \[
  \Phi(\sum c_1^m \otimes b)
  := \sum x_1^m (\pi_*)^*(b)
  \]
  for $b \in H^*(B\G_1(S^2,SO(3));\Z)$. Actually, $\Phi$ is also a ring isomorphism, since
  \begin{align*}
    \Phi(c_1^{m+m'}\otimes (bb'))
    &= x_1^{m+m'} (\pi_*)^*(bb') \\
    &= (x_1^m (\pi_*)^*(b)) (x_1^{m'} (\pi_*)^*(b')) \\
    &= \Phi(c_1^m \otimes b) \Phi(c_1^{m'} \otimes b')
  \end{align*}
  for $m, m' \ge 0$ and $b, b' \in H^*(B\G_1(S^2,SO(3));\Z)$. Since $\Phi(c_1 \otimes 1) = x_1$, we have
  \[
    H^*(B\G_1(S^2,SO(3));\Z)
    \cong
    \bigl(\Z[c_1, c_2, x_1, x_2, \ldots]/(h_2, h_3, \ldots)\rat
    \bigr)/(x_1)
  \]
  and the proof for $n = 3$ is done.
\end{proof}

Next, we compute the $n=4$ case. Since the diagram in Lemma \ref{diag2} is natural with respect to the inclusions $i'_n\colon Spin^c(n) \to Spin^c(n+1)$ and $i_n\colon SO(n) \to SO(n+1)$, we have a homotopy commutative diagram
\[
  \xymatrix{
  BS^1 \ar[r]^-{\alpha_*} \ar@{=}[d] & B\G_1(S^2,Spin^c(3)) \ar[r]^-{\pi_*} \ar[d]^{(i'_3)_*} & B\G_1(S^2,SO(3)) \ar[d]^{(i_3)_*} \\
  BS^1 \ar[r]^-{\alpha_*}  & B\G_1(S^2,Spin^c(4)) \ar[r]^-{\pi_*} & B\G_1(S^2,SO(4))
  }
\]
where all rows are homotopy fibrations.

\begin{proof}[Proof of Theorem \ref{main1} (2)]
By Proposition \ref{BGSpinc4}, we have an isomorphism
\[
  H^*(B\G_k(S^2, Spin^c(4));\Z)
  \cong
  \Z[c_1,c_2,x_1,x_2,\ldots]\otimes\Z[c_1,c_2,x_1,x_2,\ldots]/I
\]
where
\[
  I=(h_2\otimes 1,1\otimes h_2,h_3\otimes 1,1\otimes h_3,\ldots)_\mathcal{R}+(c_1\otimes 1-1\otimes c_1).
\]
Since the composite $Spin^c(3) \xrightarrow{i'_3} Spin^c(4) \hookrightarrow U(2) \times U(2)$ is the diagonal map, we have
\[
  ((i'_3)_*)^*(c_i \otimes 1) = c_i
  \quad \text{and} \quad
  ((i'_3)_*)^*(x_i \otimes 1) = x_i.
\]
Thus by Proposition \ref{alpha}, we get
\[
  (\alpha_*)^*(x_1 \otimes 1)
  = (\alpha_*)^* \circ ((i'_3)_*)^*(x_1 \otimes 1)
  = (\alpha_*)^*(x_1)
  = c_1.
\]
Then by arguing as in the proof of Theorem \ref{main1} (1), we get
\[
H^*(B\G_1(S^2, SO(4));\Z)
\cong H^*(B\G_1(S^2, Spin^c(4));\Z)/(x_1 \otimes 1),
\]
and the proof for $n = 4$ is complete.
\end{proof}

\begin{corollary}
  \label{le4}
  The homology of $B\G_1(S^2,SO(n))$ is torsion free for $n\le 4$.
\end{corollary}

\begin{proof}
  For $n \le 4$, there is a module isomorphism
  \[
  \Phi \colon H^*(BS^1;\Z) \otimes H^*(B\G_1(S^2,SO(n));\Z) \to H^*(B\G_1(S^2,Spin^c(n));\Z)
  \]
  as seen in the proof of Theorem \ref{main1}. Then $H^*(B\G_1(S^2,SO(n));\Z)$ is a direct summand of $H^*(B\G_1(S^2,Spin^c(n));\Z)$ which is a free abelian group for $n \le 4$. Thus $H^*(B\G_1(S^2,SO(n));\Z)$ is also a free abelian group, completing the proof.
\end{proof}


\section[Torsion in the cohomology of $B\G_1(S^2,SO(n))$]{Torsion in the cohomology of $B\G_1(S^2,SO(n))$}
  \label{Tor}

This section gives a proof of Theorem \ref{main2}. We consider the mod $2$ cohomology of $B\G_1(S^2, SO(n))$. Let $i_n \colon SO(n) \to SO(n+1)$ be the inclusion.

\begin{lemma}
  \label{dim2}
  There are isomorphisms
    \begin{equation*}
      H^2(\Omega_0 SO(n); \Z/2) \cong
      \begin{cases}
        \langle y_1 \rangle & (n = 3)\\
        \langle z_1, y_1 \rangle & (n = 4)\\
        \langle z_1 \rangle & (n \geq 5)
      \end{cases}
  \end{equation*}
  \noindent
  such that $(\Omega i_3)^*(y_1) = y_1, (\Omega i_3)^*(z_1) = 0$ and $(\Omega i_n)^*(z_1) = z_1$ for $n \geq 4$.
\end{lemma}

\begin{proof}
  Since there is a natural homotopy equivalence $\Omega_0 SO(n) \simeq \Omega Spin(n)$, we prove the statement for $\Omega Spin(n)$. Let $j_n \colon Spin(n) \to Spin(n+1)$ be the inclusion. There is a commutative diagram
  \[
    \xymatrix{
    Spin(3) \ar[r]^{j_3} \ar@{=}[d] & Spin(4) \ar[r]^{j_4} \ar@{=}[d] & Spin(5) \ar@{=}[d]\\
    Sp(1) \ar[r]^-{\Delta} & Sp(1) \times Sp(1) \ar[r]^-{\rho} & Sp(2)
    }
  \]
  where $\Delta$ is the diagonal map and $\rho$ is the diagonal inclusion. Then there is a induced homotopy commutative diagram of the corresponding classifying spaces. Let $q \in H^4(BSp(n); \Z/2)$ be the mod 2 reduction of the first symplectic Pontryagin class.
  Then
  \begin{equation*}
    \rho^*(q) = q \times 1 + 1 \times q
    \quad
    \text{and}
    \quad
    \Delta^*(q \times 1) = q = \Delta^*(1 \times q).
  \end{equation*}
  \noindent
  We define $y_1, z_1 \in H^2(\Omega Spin(n); \Z/2)$ by
  \begin{alignat*}{2}
    y_1 &:= \sigma^2(q) &\quad &(n = 3) \\
    z_1 &:= \sigma^2(q \times 1 + 1 \times q), \quad
    y_1 := \sigma^2(q \times 1) &\quad &(n = 4) \\
    z_1 &:= \sigma^2(q) &\quad &(n \geq 5),
  \end{alignat*}
  \noindent
  where $\sigma^2 \colon H^4(BSpin(n); \Z/2) \to H^2(\Omega Spin(n); \Z/2)$ denotes the cohomology double suspension. Since $BSpin(n)$ is $3$-connected, $\sigma^2$ is an isomorphism. Moreover, we have $(\Omega j_n) \circ \sigma^2 = \sigma^2 \circ j_n$. Thus we get $(\Omega j_3)^*(y_1) = y_1$ and $(\Omega j_3)^*(z_1) = 0$. Since $\Omega j_n$ is a $(n-2)$-equivalence, we also get $(\Omega i_n)^*(z_1) = z_1$ for $n \geq 4$, completing the proof.
\end{proof}
\noindent
Let $\{E_r(n), d_r\}$ denote the mod $2$ cohomology Serre spectral sequence for the homotopy fibration
\[
\Omega_0 SO(n) \xrightarrow{\iota} B\G_1(S^2, SO(n)) \xrightarrow{e} BSO(n).
\]
Since $BSO(n)$ is simply-connected, there is an isomorphism
\[
E_2(n) \cong H^*(BSO(n); \Z/2) \otimes H^*(\Omega_0 SO(n); \Z/2).
\]
Recall that the mod $2$ cohomology of $BSO(n)$ is given by
\[
 H^*(BSO(n); \Z/2) \cong \Z/2 [w_2, w_3, \ldots w_n]
\]
where $w_i \in H^i(BSO(n); \Z/2)$ is the universal $i$-th Stiefel-Whitney class. We also recall the classical result of Bott \cite{B58}.

\begin{lemma}
  \label{Bott2}
  Let $G$ be a connected compact Lie group. Then the integral homology of $\Omega_0 G$ is a free abelian group such that $H_{\odd}(\Omega_0 G; \Z) = 0$.
\end{lemma}
\noindent
In particular, we have
\[H^{\odd}(\Omega_0 SO(n); \Z/2) = 0.
\]
Hence $E_2(n) = E_3(n)$ by degree reasons.

\begin{lemma}
  \label{ge5}
  For $n \geq 5$, there is an isomorphism
  \[
     H^{2}(B\G_1(S^2, SO(n)); \Z/2) \cong (\Z/2)^2.
  \]
\end{lemma}

\begin{proof}
  We have
  \[
  E_3^{0,2}(n) = E_2^{0,2}(n) = \langle 1 \otimes z_1 \rangle,
  \quad
  E_3^{3,0}(n) = E_2^{3,0}(n) = \langle w_3 \otimes 1 \rangle.
  \]
  Let $i_n^*\colon E_3(n+1) \to E_3(n)$ be the induced homomorphism. Then we have
  \[
  i_n^* (d_3(1 \otimes z_1)) = d_3(1 \otimes (\Omega i_n)^*(z_1))
  \quad \text{in} \quad E_3^{3,0}(n)
  \]
  for $n \geq 3$. Since $i_3^*\colon E_3^{3,0}(n+1) \to E_3^{3,0}(n)$ is an isomorphism for $n \geq 3$ and
  \[
   d_3(1 \otimes (\Omega i_3)^*(z_1)) = d_3(1 \otimes 0) = 0
   \quad \text{in} \quad E_3^{3,0}(3)
  \]
  by Lemma \ref{dim2}, we get
  \[
   d_3(1 \otimes z_1) = 0
   \quad \text{in} \quad E_3^{3,0}(n).
  \]
  Thus $d_3 \colon E_3^{0,2}(n) \to E_3^{3,0}(n)$ is the zero map and so $1 \otimes z_1$ is a permanent cycle. Since $w_2 \otimes 1 \in E_3^{2,0}(n) \cong \Z/2$ is also a permanent cycle by degree reasons, we get
  \begin{align*}
    H^{2}(B\G_1(S^2, SO(n)); \Z/2)
    &\cong E_{\infty}^{2,0}(n) \oplus E_{\infty}^{0,2}(n) \\
    &\cong E_{3}^{2,0}(n) \oplus E_{3}^{0,2}(n)
    = \langle w_2 \otimes 1, 1 \otimes z_1 \rangle
  \end{align*}
  and the proof is finished.
\end{proof}

\begin{lemma}
  \label{H2Z}
  For $n \geq 5$, there is an isomorphism
  \[
     H^{2}(B\G_1(S^2, SO(n)); \Z) \cong \Z.
  \]
\end{lemma}

\begin{proof}
  Consider the integral cohomology Serre spectral sequence for the homotopy fibration
  \[
  \Omega_0 SO(n) \xrightarrow{\iota} B\G_1(S^2, SO(n)) \xrightarrow{e} BSO(n).
  \]
  Then its $E_2$-term is isomorphic to $H^*(BSO(n); \Z) \otimes H^*(\Omega_0 SO(n); \Z)$.
  Since
  \[
  H^i(BSO(n); \Z) \cong
  \begin{cases}
    \Z & (i = 0) \\
    0 & (i = 1, 2) \\
    \Z/2 & (i = 3),
  \end{cases}
  \quad
  H^i(\Omega_0 SO(n); \Z) \cong
  \begin{cases}
    \Z & (i = 0, 2) \\
    0 & (i = 1)
  \end{cases}
  \]
  for $n \ge 5$, we have
  \[
    H^{2}(B\G_1(S^2, SO(n)); \Z)
    \cong E_{\infty}^{0,2}
    \cong E_3^{0,2} \cong \Z
  \]
  by degree reasons. Thus we get the isomorphism in the statement.
\end{proof}

\begin{proof}[Proof of Theorem \ref{main2}]
  By Lemma \ref{ge5} we have
\begin{align*}
  (\Z/2)^2
  &\cong H^{2}(B\G_1(S^2, SO(n)); \Z/2) \\
  &\cong \left( H^{2}(B\G_1(S^2, SO(n)); \Z) \otimes \Z/2 \right)
    \oplus{\rm Tor}(H^3(B\G_1(S^2, SO(n)); \Z), \Z/2)
\end{align*}
\noindent
  for $n \geq 5$. Since $H^{2}(B\G_1(S^2, SO(n)); \Z) \cong \Z$ by Lemma \ref{H2Z}, we get
\[
  {\rm Tor}(H^3(B\G_1(S^2, SO(n)); \Z), \Z/2) \cong \Z/2.
\]
Thus by the universal coefficient theorem, the integral homology of $B\G_1(S^2, SO(n))$ has $2$-torsion for $n \geq 5$. Therefore by Corollary \ref{le4}, the proof is finished.
\end{proof}

We have so far considered only $2$-torsion in the homology of $B\G_1(S^2, SO(n))$. Now we consider $p$-torsion for $p$ odd prime.

\begin{proposition}
  \label{oddp}
  For $p$ odd prime, $H^{\odd}(B\G_k(S^2, SO(n)); \Z/p) = 0$.
\end{proposition}

\begin{proof}
  Consider the mod $p$ cohomology Serre spectral sequence for the homotopy fibration
  \[
  \Omega_0 SO(n) \xrightarrow{\iota} B\G_1(S^2, SO(n)) \xrightarrow{e} BSO(n).
  \]
  Then its $E_2$-term is isomorphic to $H^*(BSO(n); \Z/p) \otimes H^*(\Omega_0 SO(n); \Z/p)$. By Lemma \ref{Bott2} we have $H^{\odd}(\Omega_0 SO(n); \Z/p) = 0$. On the other hand, it is well known that $H^{\odd}(BSO(n); \Z/p) = 0$. Then by degree reasons, the spectral sequence collapses at the $E_2$-term and the proof is done.
\end{proof}
\noindent
Therefore by applying the universal coefficient theorem, the following statement holds.

\begin{corollary}
  For $p$ odd prime, the homology of $B\G_k(S^2, SO(n))$ has no $p$-torsion.
\end{corollary}


\section[Homotopy types of $\G_k(\Sigma_g,SO(n))$]{Homotopy types of $\G_k(\Sigma_g,SO(n))$}
 \label{Sigmag}

This section proves Theorem \ref{main3}. We recall the result of Theriault \cite{Th10A}.
\begin{proposition}
  There is a homotopy decomposition
  \[
    \G_k(\Sigma_g, U(n))\simeq (\Omega U(n))^{2g} \times \G_k(S^2, U(n)).
  \]
\end{proposition}
\noindent
One can easily see that the proof in \cite{Th10A} is also applied to $\G_k(\Sigma_g, SO(n))$. Thus we get a homotopy decomposition
\begin{equation}
  \label{decomp}
  \G_k(\Sigma_g, SO(n))
  \simeq (\Z/2)^{2g} \times (\Omega Spin(n))^{2g} \times \G_k(S^2,SO(n)).
\end{equation}

\begin{proof}[Proof of Theorem \ref{main3}]
  Recall that $\Omega Spin(n)$ is simply-connected. By \eqref{decomp}, we have
  \begin{align*}
    \pi_1(\G_k(\Sigma_g, SO(n)))
    &\cong \pi_1((\Z/2)^{2g}) \times \pi_1((\Omega Spin(n))^{2g}) \times \pi_1(\G_k(S^2,SO(n))) \\
    &\cong \pi_1(\G_k(S^2,SO(n))).
  \end{align*}
  Since $\G_k(S^2, SO(n))$ is connected,
  \[
    \pi_1(\G_k(S^2,SO(n)))
    \cong \pi_2(B\G_k(S^2,SO(n)))
    \cong H_2(B\G_k(S^2,SO(n));\Z)
  \]
  by the Hurewicz theorem. Thus it is sufficient to prove
  \[
  H_2(B\G_k(S^2,SO(n));\Z)
  \cong
  \begin{cases}
    \Z^{n-2} \oplus \Z/2 &(k = 0) \\
    \Z^{n-2} &(k = 1).
  \end{cases}
  \]
  The $k = 1$ case follows from Theorem \ref{main1} and the universal coefficient theorem. Let $\{ E'_r(n), d'_r \}$ denote the integral cohomology Serre spectral sequence for the homotopy fibration of the middle row of the diagram in Lemma \ref{diag2}
  \[
  BS^1\xrightarrow{\alpha_*}B\G_0(S^2,Spin^c(n))\xrightarrow{\pi_*}B\G_0(S^2,SO(n)).
  \]
  Since $BSO(n)$ is simply-connected, there is an isomorphism
  \[
  E'_2(n) \cong H^*(B\G_0(S^2,SO(n));\Z) \otimes H^*(BS^1;\Z)
  \]
  Then, since $H^{\odd}(BS^1;\Z) = 0$, $E'_2(n) = E'_3(n)$ by degree reasons. Let $c_i, x_i \in H^{2i}(B\G_0(S^2,Spin^c(3));\Z)$ be the generators as in Corollary \ref{BU2} and $c_1 \in H^2(BS^1;\Z)$ be the first Chern class. For $n = 3$, we have
  \[
    (\alpha_*)^*(x_1) = 0
  \]
  by Proposition \ref{alpha}. Since $(\alpha_*)^*(c_1) = 2 c_1$ by Lemma \ref{alphac}, there is an element $x \in H^2(B\G_0(S^2, SO(3));\Z)$ such that
  \[
  (\pi_*)^*(x) = x_1
  \]
  and there is an element $w \in H^3(B\G_0(S^2, SO(3));\Z)$ of order $2$ such that
  \[
  d_3(1 \otimes c_1) = w \otimes 1.
  \]
  By degree reasons, we have
  \[
  H^i(B\G_0(S^2, SO(3)); \Z)
  \cong
  \begin{cases}
    \langle x \rangle \cong \Z & (i = 2) \\
    \langle w \rangle \cong \Z/2 & (i = 3).
  \end{cases}
  \]
  Thus by the universal coefficient theorem, we get
  \[
  H_2(B\G_0(S^2, SO(3)); \Z) \cong \Z \oplus \Z/2.
  \]
  For the $n = 4$ case, we get
  \[
  H^i(B\G_0(S^2, SO(4)); \Z)
  \cong
  \begin{cases}
    \Z^2 & (i = 2) \\
    \Z/2 & (i = 3)
  \end{cases}
  \]
  by arguing as in the $n = 3$ case. Thus
  \[
  H_2(B\G_0(S^2, SO(4)); \Z) \cong \Z^2 \oplus \Z/2
  \]
  and the proof is complete.
\end{proof}


\begin{thebibliography}{99}
  \bibitem{AB83} M. F. Atiyah and R. Bott, The Yang-Mills equations over Riemann surfaces, \textit{Philos. Trans. R. Soc. London, Ser. A} {\bf 308} (1983), 523–615.

  \bibitem{B58} R. Bott, The space of loops on a Lie group, \textit{Michigan Math. J.} {\bf 5} (1958), no. 1, 35-61.

  \bibitem{Ch06} Y. Choi, Homology of the classifying space of $Sp(n)$ Gauge groups, \textit{Isr. J. Math.} {\bf 151} (2006), 167–177.

  \bibitem{CrS00} M.C. Crabb and W.A. Sutherland, Counting homotopy types of gauge groups, \textit{Proc. London Math. Soc.} {\bf 81} (2000), no. 3, 747-768.

  \bibitem{Cu18} T. Cutler, The homotopy types of $U(n)$-gauge groups over $S^4$ and $\C P^2$, \textit{Homol. Homotopy Appl.} {\bf 20} (2018), no. 1. 5–36.

  \bibitem{G72} D. H. Gottlieb, Applications of bundle map theory, \textit{Trans. Am. Math. Soc.} {\bf 171} (1972), 23-50.

  \bibitem{HamK06} H. Hamanaka and A. Kono, Unstable $K^1$-group and homotopy type of certain gauge groups, \textit{Proc. R. Soc. Edinb. A: Math.} {\bf 136} (2006), no. 1, 149-155.

  \bibitem{HamK07} H. Hamanaka and A. Kono, Homotopy type of gauge groups of $SU(3)$-bundles over $S^6$, \textit{Topol. Appl.} {\bf 154} (2007), no. 7, 1377–1380.

  \bibitem{HasKKS16} S. Hasui, D. Kishimoto, A. Kono, T. Sato, The homotopy types of $PU(3)$- and $PSp(2)$-gauge groups, \textit{Algebraic Geom. Topol.} {\bf 16} (2016), no. 3, 1813–1825.

  \bibitem{KameKT} M. Kameko, D. Kishimoto, M. Takeda, Homotopy types of gauge groups over Riemann surfaces, to appear.

  \bibitem{KamiKKKT07} Y. Kamiyama, D. Kishimoto, A. Kono and S. Tsukuda, Samelson products of $SO(3)$ and applications. \textit{Glasgow Math. J.} {\bf 49} (2007), no. 2, 405-409.

  \bibitem{KishiK10} D. Kishimoto and A. Kono, Splitting of gauge groups, \textit{Trans. Am. Math. Soc.} {\bf 362} (2010), no. 12, 6715–6731.

  \bibitem{KishiK19} D. Kishimoto and A. Kono, On the homotopy types of $Sp(n)$ gauge groups, \textit{Algebraic Geom. Topol.} {\bf 19} (2019), no. 1, 491–502.

  \bibitem{KishiKT14} D. Kishimoto, A. Kono and M. Tsutaya, On $p$-local homotopy types of gauge groups, \textit{Proc. R. Soc. Edinb. A: Math.} {\bf 144} (2014), no. 1, 149-160.

  \bibitem{KishiTh22} D. Kishimoto and S. Theriault, The mod-$p$ homology of the classifying spaces of certain gauge groups, \textit{Proc. R. Soc. Edinb. A: Math.} (2022), 1-13.

  \bibitem{KishiThT17} D. Kishimoto, S. Theriault and M. Tsutaya, The homotopy types of $G_2$-gauge groups, \textit{Topol. Appl.} {\bf 228} (2017), 92-107.

  \bibitem{KishiTsuk16} D. Kishimoto and M. Tsutaya, Infiniteness of $A_{\infty}$-types of gauge groups, \textit{J. Topol.} {\bf 9} (2016), 181-191.

  \bibitem{Ko91} A. Kono, A note on the homotopy type of certain gauge groups, \textit{Proc. R. Soc. Edinb. A: Math.} {\bf 117} (1991), no. 3-4, 295-297.

  \bibitem{M91} G. Masbaum, On the cohomology of the classifying space of the gauge group over some $4$-complexes, \textit{Bull. Soc. Math. Fr.} {\bf 119} (1991), no. 1, 1-31.

  \bibitem{S10} S. B. Smith, The homotopy theory of function spaces: a survey, \textit{Homotopy theory of function spaces and related topics} (American Mathematical Society, Providence RI, 2010), 3–39.

  \bibitem{Ta21} M. Takeda, Cohomology of the classifying spaces of $U(n)$-gauge groups over the $2$-sphere, \textit{Homol. Homotopy Appl.} {\bf 23} (2021), no. 1, 17-24.

  \bibitem{Th10A} S. Theriault, Odd primary homotopy decompositions of gauge groups, \textit{Algebraic Geom. Topol.} {\bf 10} (2010), 535–564.

  \bibitem{Th10B} S. Theriault, The homotopy types of $Sp(2)$-gauge groups, \textit{Kyoto J. Math.} {\bf 50} (2010), no. 3, 591–605.

  \bibitem{Th15} S. Theriault, The homotopy types of $SU(5)$-gauge groups, \textit{Osaka J. Math.} {\bf 52} (2015), no. 1, 15–29.

  \bibitem{Th19} S. Theriault, Homotopy decompositions of the classifying spaces of pointed gauge groups, \textit{Pac. J. Math.} {\bf 300} (2019), no. 1, 215-231.

  \bibitem{Tsuk97} S. Tsukuda, On the cohomology of the classifying space of a certain gauge group, \textit{Proc. R. Soc. Edinb. A: Math.} {\bf 127} (1997), no. 2, 407-409.

  \bibitem{Tsut12} M. Tsutaya, Finiteness of $A_n$-equivalence types of gauge groups, \textit{J. London Math. Soc.} {\bf 85} (2012), 142-164.

\end{thebibliography}
\end{document}